\documentclass[11pt, reqno]{amsart}
\usepackage{indentfirst, amssymb, amsmath, amsthm, mathrsfs, setspace, indentfirst, enumerate,  mathrsfs, amsmath, amsthm}
\usepackage[bookmarksnumbered, colorlinks, plainpages]{hyperref}
\usepackage{pdfpages}
\usepackage{mathrsfs}
\usepackage{cite}
\usepackage{graphicx}
\usepackage{float}
\newtheorem*{theo4.A}{Theorem 4.A}
\newtheorem*{theo4.B}{Theorem 4.B}

\newtheorem*{theoA}{Theorem 1.A}
\newtheorem*{theoB}{Theorem 1.B}
\newtheorem*{theoC}{Theorem 1.C}

\newtheorem*{theo2.A}{Theorem 2.A}
\newtheorem*{theo2.B}{Theorem 2.B}

\newtheorem*{theo3.A}{Theorem 3.A}
\newtheorem*{theo3.B}{Theorem 3.B}

\newtheorem*{que3.1}{Question 3.1}
\newtheorem*{cor A}{Corollary A}
\newtheorem*{cor B}{Corollary B}
\newtheorem*{Que1}{Question 1.1}
\newtheorem*{Que2}{Question 1.2}

\newtheorem{theo}{Theorem}[section]
\newtheorem{lem}{Lemma}[section]

\newtheorem{exm}{Example}[section]
\newtheorem{defi}{Definition}[section]
\newtheorem{rem}{Remark}[section]

\newcommand{\ol}{\overline}
\newcommand{\be}{\begin{equation}}
\newcommand{\ee}{\end{equation}}
\newcommand{\beas}{\begin{eqnarray*}}
\newcommand{\eeas}{\end{eqnarray*}}
\newcommand{\bea}{\begin{eqnarray}}
\newcommand{\eea}{\end{eqnarray}}

\numberwithin{equation}{section}
\begin{document}
\title[R\MakeLowercase{efined} B\MakeLowercase{ohr Inequalities and } B\MakeLowercase{ohr Radii involving}
S\MakeLowercase{chwarz mappings for Holomorphic functions in }B\MakeLowercase{anach Spaces and }...]{
\LARGE R\MakeLowercase{efined} B\MakeLowercase{ohr Inequalities and } B\MakeLowercase{ohr Radii involving}
S\MakeLowercase{chwarz mappings for Holomorphic functions in }B\MakeLowercase{anach Spaces and }JB$^*$-\MakeLowercase{triples}
}
\date{}
\author[N. S\MakeLowercase{arkar} and P. D\MakeLowercase{as}]{N\MakeLowercase{abadwip} S\MakeLowercase{arkar} \MakeLowercase{and} P\MakeLowercase{radip} D\MakeLowercase{as}}

\address{Amity School of Applied Sciences, Amity University Mumbai, Panvel, Navi Mumbai, Maharashtra-410206, India.}
\email{nsarkar@mum.amity.edu, nabadwipsarkar52@gmail.com}
\address{Department of Mathematics, Raiganj University, Raiganj, West Bengal-733134, India.}
\email{pradipsmath@gmail.com}

\renewcommand{\thefootnote}{}
\footnote{*\emph{Corresponding Author}: Pradip Das}
\renewcommand{\thefootnote}{\arabic{footnote}}
\setcounter{footnote}{0}

\begin{abstract}
In this paper, we investigate refined Bohr-type inequalities involving Schwarz mappings for holomorphic mappings defined on the unit ball \( B_X \) of a complex Banach space. We first establish refined Bohr inequalities for scalar-valued holomorphic mappings \( F : B_X \to \mathbb{U} \). We then extend some results to vector-valued holomorphic mappings involving Schwarz mappings of the form \( F : B_X \to \overline{\mathbb{U}}^{\,n} \), \( n \ge 2 \), where improved Bohr radii are obtained under suitable norm conditions.

Furthermore, we study the Bohr radius involving Schwarz mappings for holomorphic mappings whose ranges are contained in the closure of the unit ball of a JB$^*$-triple. Our results generalize the results of Jia \textit{et al.} \cite{JLP}, Liu \textit{et al.} \cite{LLP2021} and Liu \textit{et al.} \cite{LSX2018}. 
\end{abstract}

\thanks{Typeset by \AmS -\LaTeX}
\maketitle
\section{{\bf Introduction}}

Let $H^\infty$ denote the space of all bounded holomorphic functions on the unit disk $\mathbb{U}=\{z\in\mathbb{C}:\ |z|<1\}.$ Equipped with the norm $\|f\|_\infty:=\sup_{z\in\mathbb{D}}|f(z)|,$ the space $H^\infty$ is a Banach space.\par

In 1914, Harald Bohr \cite{Bohr1914} established a fundamental result asserting that for any function $f(z)=\sum_{n=0}^{\infty} a_n z^n$ belonging to the space $H^\infty$, then the inequality
\[
|a_0|+\sum_{n=1}^{\infty}|a_n|\, r^n \le \|f\|_\infty
\]
for \( 0\le r\le \frac{1}{6} \), where the coefficients are given by \( a_n=f^{(n)}(0)/n! \) for \( n\ge 0 \). Subsequently, Riesz, Schur and Wiener independently proved that this inequality holds for the wider range $0\le r\le \tfrac13$, and the number $\tfrac13$ cannot be improved. This inequality is known as the classical Bohr inequality and the constant $\tfrac13$ is called the \emph{Bohr radius}, as the function
\[
\varphi_a(z)=\frac{a-z}{1-az}\qquad(|a|<1),
\]
from the group of automorphisms of the unit disk shows.

Although Bohr's original paper \cite{Bohr1914} already includes Wiener's argument establishing that the Bohr radius equals $\tfrac13$, numerous alternative proofs have since been developed. For the classical planar setting, comprehensive accounts may be found in the survey articles \cite{AAP2017,GMR2018} and the references cited therein. It should be emphasized that no extremal function in $H^\infty$ attains the Bohr radius exactly equal to $\tfrac13$ (see \cite{AKP2019,GMR2018,KP2017}). In recent years, Bohr's theorem has inspired a wide range of extensions, refinements, and variants in diverse analytic contexts; see, for instance, \cite{AA2013,BD2018,BB2004,KS2022} and the references therein.

In recent years, many authors have extended this problem to functions of several complex variables (cf.~\cite{AAH2026,LLP2023,LL2020} and the references therein). One of the aims of this paper is to generalize or improve multidimensional analogues of refined Bohr inequalities involving Schwarz functions and to prove that these refinements are best possible.\par
Let $\mathcal{B}$ denotes the set of all analytic self-maps of the unit disk  $\mathbb{U}=\{z\in\mathbb{C}:\ |z|<1\}.$

More recently, Liu et. al \cite{LSX2018} obtained the following
result on the Bohr radius for analytic functions.
\begin{theoA}\cite[Lemma 2.8]{LSX2018} Suppose that  $f\in \mathcal{B}$ and $f(z)=\sum_{k=0}^{\infty} a_k z^k$. Then
\begin{equation*}\label{eq:2.15}
|f(z)|+\sum_{k=1}^{\infty} |a_{2k}|\,|z|^{2k} \le 1,
\quad \text{for } |z|=r \le \sqrt{2}-1.
\end{equation*}
Moreover, the radius $r=\sqrt{2}-1$ is best possible.

\end{theoA}
Based on the initiation of Kayumov and Ponnusamy \cite{KP1}, several forms of Bohr-type inequalities for the family $\mathcal{B}$ were considered in \cite{LSX2018} when the Taylor coefficients of the classical Bohr inequality are partly or completely replaced by higher order derivatives of $f$. Here we only recall one of them.
\begin{theoB}\cite[Theorem 2.1]{LSX2018}
Suppose that $f \in \mathcal{B}$ and $f(z)=\sum_{n=0}^{\infty} a_n z^n .$
Then the following sharp inequality holds:
\begin{equation*}\label{eq:F}
|f(z)| + |f'(z)|\, r + \sum_{n=2}^{\infty} |a_n|\, r^n \le 1,
\quad \text{for } |z|=r \le \frac{\sqrt{17}-3}{4}.
\end{equation*}
\end{theoB}
In 2021, Liu et al. \cite{LLP2021} established a refined version of Bohr's inequality, which significantly extends the classical Bohr phenomenon by incorporating sharper coefficient estimates.
\begin{theoC}\cite[Theorem 7]{LLP2021}
Suppose that $f$ is an analytic function in $\mathbb{U}$ such that $|f(z)|<1$ for $z\in \mathbb{U}$  and
\[
f(z)=\sum_{n=0}^{\infty} a_n z^n .
\]
Then
\[
I(z)
:= |f(z)| + |f'(z)|\, r
+ \sum_{n=2}^{\infty} |a_n| r^n
+ \left(\frac{1}{1+|a_0|}+\frac{r}{1-r}\right)
\sum_{n=1}^{\infty} |a_n|^2 r^{2n}
\le 1
\]
for $|z|=r\le \dfrac{\sqrt{17}-3}{4}$, and the constant $\dfrac{\sqrt{17}-3}{4}$ is best possible. Moreover 

\[
J(z)
:= |f(z)|^2 + |f'(z)|\, r
+ \sum_{n=2}^{\infty} |a_n| r^n
+ \left(\frac{1}{1+|a_0|}+\frac{r}{1-r}\right)
\sum_{n=1}^{\infty} |a_n|^2 r^{2n}
\le 1
\]
for $|z|=r\le r_0$, where $r_0 \approx 0.385795$ is the unique positive root of the equation
\[
1 - 2r - r^2 - r^3 - r^4 = 0.
\]
The constant $r_0$ is best possible.
\end{theoC}


In order to present our results in a clear and systematic manner, we first fix some
notation. For a point $z=(z_1,z_2,\ldots,z_n)\in\mathbb{C}^n$, we denote by $\|z\|_\infty:=\max_{1\le i\le n}|z_i|$
the supremum norm. The unit polydisc in $\mathbb{C}^n$ is then defined by
$\mathbb{U}^n:=\{z\in\mathbb{C}^n:\ \|z\|_\infty<1\}.$\par
Let $X$ and $Y$ be complex Banach spaces.
\begin{defi}
Let $k\in\mathbb{N}$. A mapping $P\colon X\to Y$ is called a \emph{homogeneous polynomial of degree $k$}
if there exists a $k$-linear mapping $u\colon X^k\to Y$ such that
\[
P(x)=u(x,\ldots,x), \qquad x\in X.
\]
\end{defi}

Note that if $P$ is a homogeneous polynomial of degree $k$, then
\[
P(\lambda x)=\lambda^k P(x), \qquad x\in X,\ \lambda\in\mathbb{C}.
\]
Throughout this paper, the degree of a homogeneous polynomial is indicated by a subscript.
That is, if $P_k$ is a homogeneous polynomial, then its degree is $k$.
Moreover, if $P_k$ is a $k$-homogeneous polynomial from $X$ into $Y$, then there exists a unique
symmetric $k$-linear mapping $u$ such that
\[
P_k(x)=u(x,\ldots,x), \qquad x\in X.
\]

Let $\mathcal{D}\subset X$ be a domain and let $F\colon \mathcal{D}\to Y$ be a holomorphic mapping.
For $z\in \mathcal{D}$, denote by $D^kF(z)$ the $k$-th Fr\'{e}chet derivative of $F$ at $z$.
If $\mathcal{D}$ contains the origin, then $F$ admits the expansion
\begin{equation}
F(z)=\sum_{k=0}^{\infty}\frac{1}{k!}\,D^kF(0)(z^k)
\label{eq:Taylor}
\end{equation}
in a neighbourhood of the origin.
Since $\frac{1}{k!}D^kF(0)(z^k)$ is a homogeneous polynomial of degree $k$, we shall use the notation
\[
P_k(z):=\frac{1}{k!}D^kF(0)(z^k)
\]
throughout this paper.
If $\mathcal{D}$ is a bounded balanced domain in $X$ and $F(\mathcal{D})$ is bounded, then the series
\eqref{eq:Taylor} converges uniformly on $r\mathcal{D}$ for each $r\in(0,1)$.

Let $F\colon \mathcal{D}\to Y$ be holomorphic. For $k\in\mathbb{N}$, we say that $z=0$ is a zero of order $k$ of $F$
if
\[
F(0)=0,\quad DF(0)=0,\ \ldots,\ D^{k-1}F(0)=0,\quad \text{but } D^kF(0)\neq 0.
\]

Let $B_X$ and $B_Y$ denote the open unit balls of the Banach spaces $X$ and $Y$, respectively.
A holomorphic mapping $\mu\colon B_X\to B_Y$ with $\mu(0)=0$ is called a \emph{Schwarz mapping}.
If $\mu_k$ is a Schwarz mapping such that $z=0$ is a zero of order $k$ of $\mu$, then the following estimate holds
(see, for example, \cite[Lemma~6.1.28]{GK2003}):
\begin{equation}
\|\mu_k(z)\|_Y \le \|z\|_X^{\,k}, \qquad z\in B_X.
\label{eq:Schwarz}
\end{equation}

Let $L(X,\mathbb{C})$ denote the space of continuous linear operators from $X$ into $\mathbb{C}$.
For each $x\in X\setminus\{0\}$, define
\[
T(x):=\{\,\ell_x\in L(X,\mathbb{C}) : \ell_x(x)=\|x\|,\ \|\ell_x\|=1\,\}.
\]
By the Hahn--Banach theorem, the set $T(x)$ is nonempty.\par
Motivated by the above results, we naturally pose the following problems.

\begin{Que1}
Can improved versions of Theorems 1.A and 1.C be established by involving two Schwarz functions?
\end{Que1}

\begin{Que2}
Can multidimensional analogues of Theorems 1.A and 1.C be obtained by involving two Schwarz mappings?
\end{Que2}
The main objective of this article is to provide affirmative answers to the above two problems. 
The paper is organized as follows. \par

In Section 2, we state our main theorems and present several remarks; in particular, 
Theorems~2.1--2.2 essentially provide affirmative answers to Question ~1.1 and~1.2. 
In Section~3, we establish a key lemma required for the proofs of our main results and 
also generalize the result of Jia \textit{et al.} \cite{JLP}. 
In Section~4, we further generalize Theorem 1.A to the setting of JB$^{*}$-triples, 
which again yields affirmative answers to Question 1.1 and 1.2.

\section{{\bf Refined Bohr Inequalities for Holomorphic Mappings $F : B_X \to \mathbb{U}$  in Banach Spaces.}}
In order to obtain multidimensional analogues of Theorem 1.C, we focus on
Bohr-type inequalities for holomorphic mappings defined on the unit ball $B_X$
of a complex Banach space $X$. Throughout this discussion, $X$ is assumed to be
a complex Banach space, and $\mu_n : B_X \to B_X$ denotes a Schwarz mapping with
a zero of order $n$ at the origin, where $n \in \mathbb{N}$.

We begin by presenting the multidimensional counterpart of Theorem 1.C. For
brevity, let $\text{Hol}(B_X,\mathbb{U})$ represent the collection of all holomorphic
functions $f$ from $B_X$ into $\mathbb{U}$ that admit a homogeneous polynomial expansion of the form
\begin{equation}\label{m0}
F(z)=\sum_{j=0}^{\infty} P_j(z), \qquad z \in B_X,
\end{equation}
and define $a := F(0).$

\begin{theo}\label{T0.1}
Let $F\in \text{Hol}(B_X,\mathbb{U})$ with the form (\ref{m0}).  Then for $m,k\in\mathbb{N}$, we have
\begin{equation}\label{1.1}
|F(\mu_m(z))| + |DF(\mu_m(z))(\mu_m(z))|+\sum_{s=2}^{\infty}|P_{s}(\mu_k(z))|+ \left(\frac{1}{1+|a|}+\frac{r^{k}}{1-r^{k}}\right)\sum_{s=1}^{\infty}|P_{s}(\mu_k(z))|^2\le 1.
\end{equation}
for $\|z\|=r\le R_{1,m,k}< R_{1,m}$, where $\mu_k,\mu_m:B_{X}\to B_{X}$ are Schwarz mappings having $z=0$ as a zero of order $k,m$ respectively, $R_{1,m}$ is a unique root of $r^{2m}+2r^m-1=0$ and $R_{1,m,k}$ is the unique root in $(0,1)$ of the equation
\begin{equation}\label{1.2}-\frac{1-r^m}{1+r^m}+\frac{2r^m}{(1+r^m)^2}+\frac{2r^{2k}}{1-r^k}=0.\end{equation}
The constant $R_{1,m,k}$ cannot be improved.
\end{theo}
\begin{rem} In the special case where $B_X=\mathbb{U}$ and $\mu_k(z)=\mu_m(z)=z$, 
Theorem~\ref{T0.1} coincides with the first assertion of Theorem~1.C.
\end{rem}
\begin{table}[H]
\centering
\begin{minipage}[b]{0.45\textwidth}
\centering
\begin{tabular}{|c|c|c|}
\hline
\textbf{k} & \textbf{m} & \textbf{$R_{1,m,k}$}  \\ 
\hline
1 & 1 &  0.280800 \\ 
\hline
1 & 3 & 0.441100 \\ 
\hline
2 & 5 &  0.648400 \\ 
\hline
3 & 7 &  0.744000 \\ 
\hline
\end{tabular}
\caption{$R_{1,m,k}$ is the unique root of equation (\ref{1.2}) in $(0,1)$.}
\label{tab3.1}
\end{minipage}%
\hfill
\begin{minipage}[b]{0.55\textwidth}
\centering
\begin{figure}[H]
\centering
\includegraphics[width=\textwidth]{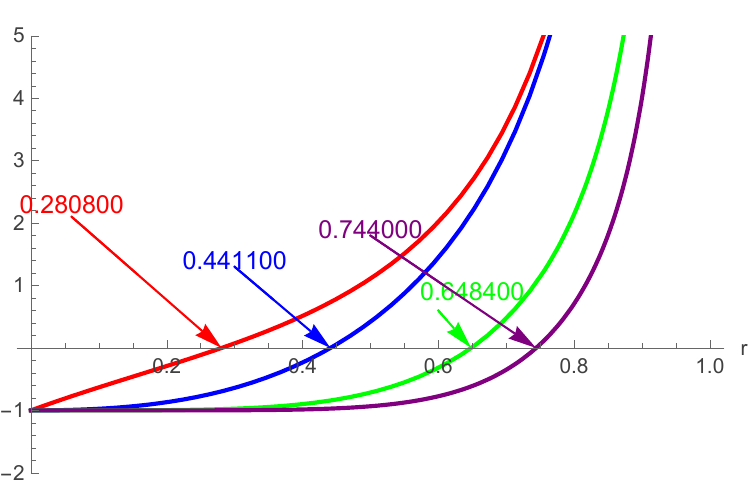}
\caption{The graphs exhibit the locations of the roots $R_{1,m,k}$ in $(0,1)$ for different values of $k,m$.}
\label{f31}
\end{figure}
\end{minipage}
\end{table}

\begin{theo}\label{T0.2}
Let $F\in \text{Hol}(B_X,\mathbb{U})$ with the form (\ref{m0}).  Then for $m,k\in\mathbb{N}$, we have
\begin{equation}\label{1.3}
|F(\mu_m(z))|^2 +|DF(\mu_m(z))(\mu_m(z))|+ \sum_{s=2}^{\infty}|P_{s}(\mu_k(z))|+ \left(\frac{1}{1+|a|}+\frac{r^{k}}{1-r^{k}}\right)\sum_{s=1}^{\infty}|P_{s}(\mu_k(z))|^2\le 1.
\end{equation}
for $\|z\|=r\le R_{2,m,k}< R_{2,m}$, where $\mu_k,\mu_m:B_{X}\to B_{X}$ are Schwarz mappings having $z=0$ as a zero of order $k,m$ respectively, $R_{2,m}$ is a unique root of $r^{2m}+r^m-1=0$ and $R_{2,m,k}$ is the unique root in $(0,1)$ of the equation
\begin{equation}\label{1.4}-\,\frac{1-r^{2m}-r^{m}}{(1+r^{m})^{2}}
+ \frac{r^{2k}}{1-r^{k}} = 0.\end{equation}
The constant $R_{2,m,k}$ cannot be improved.
\end{theo}
\begin{rem} In the special case where $B_X=\mathbb{U}$ and $\mu_k(z)=\mu_m(z)=z$, 
Theorem~\ref{T0.2} coincides with the second assertion of Theorem~1.C.
\end{rem}
\begin{table}[H]
\centering
\begin{minipage}[b]{0.45\textwidth}
\centering
\begin{tabular}{|c|c|c|}
\hline
\textbf{k} & \textbf{m} & \textbf{$R_{2,m,k}$}  \\ 
\hline
1 & 1 & 0.385800 \\ 
\hline
1 & 3 & 0.535700 \\ 
\hline
2 & 5 &  0.717500 \\ 
\hline
3 & 7 &  0.797100 \\ 
\hline
\end{tabular}
\caption{$R_{2,m,k}$ is the unique root of equation (\ref{1.4}) in $(0,1)$.}
\label{tab3.2}
\end{minipage}%
\hfill
\begin{minipage}[b]{0.55\textwidth}
\centering
\begin{figure}[H]
\centering
\includegraphics[width=\textwidth]{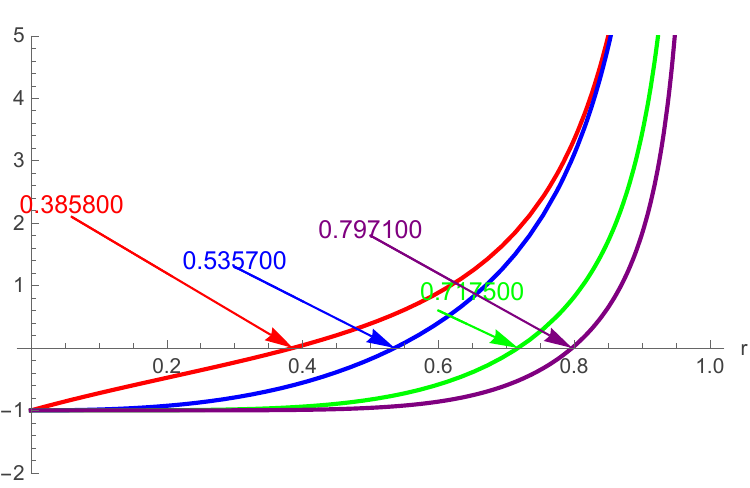}
\caption{The graphs exhibit the locations of the roots $R_{2,m,k}$ in $(0,1)$ for different values of $k,m$.}
\label{f32}
\end{figure}
\end{minipage}
\end{table}

The following are key lemmas of this paper and will be used to prove the main results.

\begin{lem}[Schwarz--Pick Lemma]\label{lem:Schwarz-Pick}
Let $\phi(z)=\sum_{n=0}^{\infty} a_n z^n$
be an analytic function in the unit disk $\mathbb{D}=\{z\in\mathbb{C}: |z|<1\},$
such that $|\phi(z)|<1$ for all $z\in\mathbb{D}$. Then the following assertions hold:

\begin{enumerate}
\item For any $z_1, z_2 \in \mathbb{D}$,
\[
\frac{|\phi(z_1)-\phi(z_2)|}{|1-\phi(z_1)\overline{\phi(z_2)}|}
\le
\frac{|z_1-z_2|}{|1-z_1\overline{z_2}|}.
\]
Moreover, equality holds for distinct $z_1, z_2 \in \mathbb{D}$ if and only if
$\phi$ is a M\"{o}bius transformation of $\mathbb{D}$.

\item For any $z\in\mathbb{D}$,
\[
|\phi'(z)| \le \frac{1-|\phi(z)|^2}{1-|z|^2}.
\]
Furthermore, equality holds for some $z\in\mathbb{D}$ if and only if
$\phi$ is a M\"{o}bius transformation.
\end{enumerate}
\end{lem}

\begin{lem}\label{L1} \cite{CHPV}
Suppose that $B_{X}$ and $B_{Y}$ are the unit balls of the complex Banach spaces $X$ and $Y$, respectively. Let $F: B_{X} \to B_{Y}$ be a holomorphic mapping. Then
\begin{equation*}
\|F(z)\|_{Y} \le \frac{\|F(0)\|_{Y} + \|z\|_{X}}{1 + \|F(0)\|_{Y}\|z\|_{X}}, \qquad z \in B_{X}.
\end{equation*}
This estimate is sharp with equality possible for each value of $\|f(0)\|_{Y}$ and for each $z \in B_{X}$.
\end{lem}
\begin{lem}\label{L2}\cite{LLP2021}
Suppose that $f$ is an analytic self-maps of the unit disk $\mathbb{U}$ defined by
\[
f(z)=\sum_{n=0}^{\infty} a_n z^n .
\]
Then, for any $N \in \mathbb{N}$, the following inequality holds:
\[
\sum_{n=N}^{\infty} |a_n|\, r^n
+ \operatorname{sgn}(t)\sum_{n=1}^{t} |a_n|^2\,\frac{r^{N}}{1-r}
+ \left(\frac{1}{1+|a_0|}+\frac{r}{1-r}\right)
\sum_{n=t+1}^{\infty} |a_n|^2 r^{2n}
\leq (1-|a_0|^2)\,\frac{r^{N}}{1-r},
\]
for $r \in [0,1)$, where $t=\left\lfloor \frac{N-1}{2} \right\rfloor .$
\end{lem}

\section*{{\bf Proof of the Main Results}}
\begin{proof}[{\bf Proof of Theorem \ref{T0.1}}]  Let $z\in B_X\setminus\{0\}$ be fixed and define $z_0=\frac{z}{\|z\|_X}\in \partial B_X.$ Also let $f(\zeta)=F(\zeta z_{0})$. Then $f$ is a holomorphic functions in $\mathbb{U}$ and
\bea\label{te1.0}
f(\zeta)=a+\sum_{s=1}^{\infty}(P_{s})(z_{0})\zeta^{s},\qquad \zeta\in\mathbb{U}.
\eea
Therefore,  we have (see e.g. \cite[p.\,35]{GK2003}), $|P_{s}(z_{0})|\leq 1-|a|^2$ for all $s\geq 1$. Now applying Lemma \ref{L2}, we deduce that 

\bea\label{te1.1} &&\sum_{s=2}^{\infty}|P_{s}(z_{0})|\,|\zeta|^{s}
+\left(\frac{1}{1+|a|}+\frac{|\zeta|}{1-|\zeta|}\right)
\sum_{s=1}^{\infty} |P_{s}(z_{0})|^2 |\zeta|^{2s}\nonumber\\
& \le & (1-|a|^2)\frac{|\zeta|^2}{1-|\zeta|}.\eea
Let $\zeta=\|z\|_X=r<1$. Then by Eq. (\ref{te1.1}), we deduce that 
\bea\label{te1.2}\sum_{s=2}^{\infty}|P_{s}(z)|+\left(\frac{1}{1+|a|}+\frac{\|z\|_X}{1-\|z\|_X}\right)
\sum_{s=1}^{\infty} |P_{s}(z)|^2 \leq (1-|a|^2)\frac{\|z\|_X^2}{1-\|z\|_X}.\eea

Note that $t(x)=\frac{x}{1-x}$. Then $t'(x)=\frac{1}{(1-x)^2}>0,$ so $t(x)$ is increasing on $(0,1)$, since $r^k\leq r$ for $k\geq 1$, it follows that 
\[t(r^k)\leq t(r).\]
Hence 
\[\frac{r^k}{1-r^k}\leq \frac{r}{1-r} =\frac{\|z\|_X}{1-\|z\|_X}.\]
That is 
\bea\label{ne1}\frac{1}{1+|a|}+\frac{r^k}{1-r^k}\leq \frac{1}{1+|a|}+\frac{\|z\|_X}{1-\|z\|_X}\eea
Therefore from (\ref{te1.2}) and (\ref{ne1}), we get 
\beas &&\sum_{s=2}^{\infty}|P_{s}(z)|+\left(\frac{1}{1+|a|}+\frac{r^k}{1-r^k}\right)
\sum_{s=1}^{\infty} |P_{s}(z)|^2\\
&\leq &\sum_{s=2}^{\infty}|P_{s}(z)|+\left(\frac{1}{1+|a|}+\frac{\|z\|_X}{1-\|z\|_X}\right)
\sum_{s=1}^{\infty} |P_{s}(z)|^2 \leq (1-|a|^2)\frac{\|z\|_X^2}{1-\|z\|_X}\eeas
and so 
\bea\label{ne2} \sum_{s=2}^{\infty}|P_{s}(z)|+\left(\frac{1}{1+|a|}+\frac{r^k}{1-r^k}\right)
\sum_{s=1}^{\infty} |P_{s}(z)|^2\leq (1-|a|^2)\frac{\|z\|_X^2}{1-\|z\|_X}.\eea

In view of (\ref{ne2}) and (\ref{eq:Schwarz}), we obtain
\bea\label{te1.3}&&\sum_{s=2}^{\infty}|P_{s}(\mu_k(z))|+\left(\frac{1}{1+|a|}+\frac{r^k}{1-r^k}\right)
\sum_{s=1}^{\infty} |P_{s}(\mu_k(z))|^2\nonumber\\
 &\leq & (1-|a|^2)\frac{\|\mu_k(z)\|_X^2}{1-\|\mu_k(z)\|_X}\;\left(\text{as $\frac{x}{1-x}$ is increasing in $(0,1)$}\right)\nonumber\\
&\leq &(1-|a|^2)\frac{\|z\|_X^{2k}}{1-\|z\|_X^k}.\eea
Differentiating (\ref{te1.0}), we have
\[f'(\zeta)=D F(\zeta z_{0}) z_0.\]
Then by Schwarz-Pick Lemma, we can conclude that 
\[ |D F(\zeta z_{0}) z_0|=|f'(\zeta)|\leq \frac{1-|f(\zeta)|^2}{1-|\zeta|^2}.\]
Let $\zeta=\|z\|_X=r<1$. Then the above Eq. becomes
\bea\label{te1.4} |D F(z) z|\leq \left(\frac{1-|F(z)|^2}{1-\|z\|^2}\right) \|z\|_X.\eea
By combining \eqref{te1.4} and \eqref{eq:Schwarz}, we obtain

\bea\label{te1.5} |D F(\mu_m(z)) \mu_m(z)|\leq \left(\frac{1-|F(\mu_m(z))|}{1-\|\mu_m(z)\|^2}\right) \|\mu_m(z)\|_X\leq \left(\frac{1-|F(\mu_m(z))|}{1-\|z\|_X^{2m}}\right) \|z\|_X^m.\eea
Let $x=|a|\in[0,1]$. If $x=1$, then $P_s(z_0)=0$ for all $s\ge1$, and
\eqref{1.1} holds trivially. Hence, we assume $x\in[0,1)$.  Then from \eqref{te1.4} and \eqref{te1.5}, we obtain
\bea\label{te1.6} &&|F(\mu_m(z))|+|D F(\mu_m(z)) \mu_m(z)|+\sum_{s=2}^{\infty}|P_{s}(\mu_k(z))|+\left(\frac{1}{1+|a|}+\frac{r^k}{1-r^k}\right)
\sum_{s=1}^{\infty} |P_{s}(\mu_k(z))|^2\nonumber\\
&&\leq |F(\mu_m(z))|+\left(\frac{1-|F(\mu_m(z))|^2}{1-r^{2m}}\right) r^m+(1-x^2)\frac{r^{2k}}{1-r^k}\nonumber\\
&&\leq  \frac{r^{m}+x}{1+r^{m}x}+\left(1-\left(\frac{r^{m}+x}{1+r^{m}x}\right)^2\right)\frac{r^m}{1-r^{2m}}+(1-x^2)\frac{r^{2k}}{1-r^k}\nonumber\\
&&=1+\frac{(1-r^m)(x-1)}{1+xr^m}+\frac{(1-x^2)(1-r^{2m})}{(1+xr^m)^2}\frac{r^m}{1-r^{2m}}+(1-x^2)\frac{r^{2k}}{1-r^k}\nonumber\\
&&=1+\frac{(1-x)D_1(x,r)}{1+xr^m},\eea
where 
\[D_1(x,r)=-(1-r^m)+\frac{(1+x)r^m}{1+xr^m}+\frac{r^{2k}(1+x)(1+xr^m)}{1-r^k}.\]
It is worth pointing out that, in the third inequality of Eq. (\ref{te1.6}), we have used the fact that
\begin{equation}\label{ll1}
\Phi(t)=t+\lambda(1-t^{2})\leq \Phi(t_{0}),
\end{equation}
whenever
\begin{equation*}\label{ll2}
t=|F(\mu_m(z))|\leq t_{0}=\frac{r^{m}+x}{1+r^{m}x}
\quad \text{and} \quad
\lambda=\frac{r^{m}}{1-r^{2m}}.
\end{equation*}
Inequality \eqref{ll1} holds provided that $\frac{2r^{m}}{1-r^{2m}}\leq 1,$ which is satisfied for $0\leq r\leq R_m$, where $R_m$ is the unique positive root of the equation $r^{2m}+2r^{m}-1=0.$\par
Differentiating partially $D_1(x,r)$ with respect to $x$, we deduce that 
\[\frac{\partial D_1(x,r)}{\partial x}=\frac{(1-r^m)r^m}{(1+xr^m)^2}+\frac{r^{2k}(1+r^m+2xr^m)}{1-r^k}\geq 0,\]
which shows that $D_1(x,r)$ is increasing with respect to $x$ and so we deduce that 
\beas D_1(x,r)\leq D_1(1,r)=-(1-r^m)+\frac{2r^m}{1+r^m}+\frac{2r^{2k}(1+r^m)}{1-r^k}\leq 0,\eeas
for $r\leq R_{1,m,k}$, where $R_{1,m,k}$ is minimal positive root of (\ref{1.2}). We can see that 
\beas
\frac{d}{dr}D_1(1,r)&=&mr^{m-1}+\frac{2mr^{m-1}}{(1+r^m)^2}+\frac{2(2kr^{2k-1}(1+r^m)+mr^{2k+m-1})}{1-r^k}\\
&&+\frac{2kr^{3k-1}(1+r^m)}{(1-r^k)^2}\geq 0.
\eeas
and so $D_1(1,r)$ is increasing in $(0,1)$. Note that $D_1(1,0)=-1$ and $\lim_{r\rightarrow 1^-}D_1(1,r)=+\infty$. Therefore $R_{1,m,k}$ is unique  root of (\ref{1.2}) in $(0,1)$. \par 
We now claim that \( R_{1,m,k} < R_{1,m} \).
Suppose, on the contrary that \( R_{1,m,k} \geq R_{1,m} \).
Then for any \( r \geq R_{1,m} \) we have $-1+2r^{m}+r^{2m}\geq 0,$
and hence
\begin{align*}
-(1-r^m)+\frac{2r^m}{1+r^m}+\frac{2r^{2k}(1+r^m)}{1-r^k}
&=
\frac{-1+2r^{m}+r^{2m}}{1+r^{m}}
+\frac{2r^{2k}(1+r^{m})}{1-r^{k}}  >0.
\end{align*}
This contradicts \( R_{1,m,k} \) is a root of (\ref{1.2}).
Therefore $R_{1,m,k} < R_{1,m}.$\par

\medskip

{\bf Sharpness of the Radius $R_{1,m,k}$:}
Next, we prove the sharpness for $R_{1,m,k}$. We consider the functions
 \bea\label{fa}
f_a(z)
= \frac{a + T_v(z)}{1 + a\,T_v(z)}
= a + (1-a^2)\sum_{s=1}^{\infty}(-a)^{\,s-1}\bigl(T_v(z)\bigr)^s,
\qquad z \in B_X,\ a \in (0,1)
\eea
and $\mu_n(z)=T_v(z)^{n-1}z$, where $T_v\in T(v)$ and $v\in \partial B_X$ for $n\geq 1$.
Clearly 
\[ Df_a(z)(z) =\frac{(1-a^2)\,T_v(z)}{\bigl(1+a\,T_v(z)\bigr)^2}.\]

For \(z = r v\), where \(v \in \partial B_X\), we have $T_v(z) = r$ and so $\mu_n(z) = r^n v,  n \in \mathbb{N}$ and $P_s(\mu_n(z))=(1-a^2)(-a)^{s-1}r^{sn}$ for $s\geq 1$.\par
Therefore, by a simple computation it follows that
\beas &&\mathcal{A}(f_a(z)):=|f_a(\mu_m(r v))|+|D f_a(\mu_m(r v)) \mu_m(r v)|+\sum_{s=2}^{\infty}|P_{s}(\mu_k(r v))|\\
&&+\left(\frac{1}{1+|a|}+\frac{r^k}{1-r^k}\right)
\sum_{s=1}^{\infty} |P_{s}(\mu_k(r v))|^2\\
&=&\frac{a+r^m}{1+ar^m}+\frac{(1-a^2)r^m}{\bigl(1+ar^m\bigr)^2}+(1-a^2)\sum_{s=2}^{\infty}a^{s-1}r^{sk}+\frac{(1-a^2)^2(1+|a|\,r^{k})}{(1+|a|)(1-r^{k})}\sum_{s=1}^{\infty}a^{2s-2}r^{2sk}\\
&=&1-\frac{(1-a)(1-r^m)}{1+ar^m}+\frac{(1-a^2)r^m}{\bigl(1+ar^m\bigr)^2}+(1-a^2) \frac{a\, r^{2k}}{1-a r^{k}}+\frac{(1-a^2)^2(1+|a|\,r^{k})}{(1+|a|)(1-r^{k})}\frac{r^{2k}}{1-a^{2} r^{2k}}\\
&=&1+(1-a) G_1(a,r),\eeas
where 
\[G_1(a,r)=-\frac{1-r^m}{1+ar^m}+\frac{(1+a)r^m}{\bigl(1+ar^m\bigr)^2}+(1+a)\frac{a\, r^{2k}}{1-a r^{k}}+\frac{(1-a^2)(1+a)(1+|a|\,r^{k})}{(1+|a|)(1-r^{k})}\frac{ r^{2k}}{1-a^{2} r^{2k}}.\]
Let $a\rightarrow 1$. It follows that 

\[\lim_{a\rightarrow 1}G_1(a,r)=\frac{1}{1+r^m}\left[-(1-r^m)+\frac{2r^m}{1+r^m}+\frac{2r^{2k}(1+r^m)}{1- r^{k}}\right]>0,\;\text{for}\;r>R_{1,m,k}.\]
Hence $\mathcal{A}(f_a(z))>1$ for $r>R_{1,m,k}$.
\end{proof}
\begin{proof}[{\bf Proof of the Theorem \ref{T0.2}}] Using similar argument as in the proof of Theorem \ref{T0.1}, we obtain (\ref{te1.3}) and
(\ref{te1.4}). Therefore, letting $x=|a|$ we have
\bea\label{te2.1} &&|F(\mu_m(z))|^2+|D F(\mu_m(z)) \mu_m(z)|+\sum_{s=2}^{\infty}|P_{s}(\mu_k(z))|+\left(\frac{1}{1+|a|}+\frac{r^k}{1-r^k}\right)
\sum_{s=1}^{\infty} |P_{s}(\mu_k(z))|^2\nonumber\\
&&\leq |F(\mu_m(z))|^2+\left(\frac{1-|F(\mu_m(z))|^2}{1-r^{2m}}\right) r^m+(1-x^2)\frac{r^{2k}}{1-r^k}\nonumber\\
&&\leq  \left(\frac{r^{m}+x}{1+r^{m}x}\right)^2\left(1-\frac{r^m}{1-r^{2m}}\right)+\frac{r^m}{1-r^{2m}}+(1-x^2)\frac{r^{2k}}{1-r^k}\nonumber\\
&&=1+\frac{1-r^m-r^{2m}}{1-r^{2m}}\left(\left(\frac{r^{m}+x}{1+r^{m}x}\right)^2-1\right)+(1-x^2)\frac{r^{2k}}{1-r^k}\nonumber\\
&&=1+(1-x^2)D_2(x,r).
\eea
The  third inequality of Eq. (\ref{te2.1}) holds for such $r \in [0,1]$ satisfying $1-\frac{r^{m}}{1-r^{2m}} \geq 0,$
that is for $r \in [0,R_{2,m}]$, where $R_{2,m} \in (0,1)$ is the unique root of the equation $1-r^{2m}-r^{m}=0$
and
\[
D_{2}(x,r)
= -\,\frac{1-r^{2m}-r^{m}}{(1+x r^{m})^{2}}
+ \frac{r^{2k}}{1-r^{k}}.
\]

It is evident that the function $D_{2}(x,r)$ is monotonically increasing with respect to $x$.
Therefore, we obtain
\[
D_{2}(x,r) \leq D_{2}(1,r)
= -\,\frac{1-r^{2m}-r^{m}}{(1+r^{m})^{2}}
+ \frac{r^{2k}}{1-r^{k}}
\leq 0
\]
for $r \leq R_{2,m,k}$, where $R_{2,m,k}$ is the smallest positive root of the equation
\begin{equation}\label{eq:3.5}
-\,\frac{1-r^{2m}-r^{m}}{(1+r^{m})^{2}}
+ \frac{r^{2k}}{1-r^{k}} = 0 .
\end{equation}

Differentiating $D_2(x,r)$ with respect to $r$, we get 
\[\frac{\partial D_2(1,r)}{\partial r}=\frac{m r^{m-1}\left(3+2r^m-r^{2m}\right)}
{(1+r^m)^3}+ \frac{k r^{2k-1}\bigl(2-r^{k}\bigr)}{(1-r^{k})^{2}} >0\]
for all $r\in (0,1)$ and so $D_2(1,r)$ is increasing in $(0,1)$. Note that $D_2(1,0)=-1$ and $\lim_{r\rightarrow 1^-}D_2(1,r)=+\infty$. Therefore $R_{2,m,k}$ is the unique root of (\ref{1.4}) in $(0,1)$.

We now claim that \( R_{2,m,k} < R_{2,m} \).
Suppose, on the contrary that \( R_{2,m,k} \geq R_{2,m} \).
Then for any \( r \geq R_{2,m} \) we have $-1+r^{m}+r^{2m}\geq 0,$
and hence
\begin{align*}
-\,\frac{1-r^{2m}-r^{m}}{(1+r^{m})^{2}}
+ \frac{r^{2k}}{1-r^{k}}  >0.
\end{align*}
This contradicts \( R_{2,m,k} \) is a root of (\ref{1.4}).
Therefore $R_{2,m,k} < R_{2,m}.$\par

\medskip

{\bf Sharpness of the Radius $R_{2,m,k}$:} Next, we prove the sharpness for $R_{2,m,k}$. We consider the function define in (\ref{fa}).
Therefore, by a simple computation it follows that
\beas &&\mathcal{B}(f_a(z)):=|f_a(\mu_m(r v))|^2+|D f_a(\mu_m(r v)) \mu_m(r v)|+\sum_{s=2}^{\infty}|P_{s}(\mu_k(r v))|\\
&&+\left(\frac{1}{1+|a|}+\frac{r^k}{1-r^k}\right)
\sum_{s=1}^{\infty} |P_{s}(\mu_k(r v))|^2\\
&=&\left(\frac{a+r^m}{1+ar^m}\right)^2+\frac{(1-a^2)r^m}{\bigl(1+ar^m\bigr)^2}+(1-a^2)\sum_{s=2}^{\infty}a^{s-1}r^{sk}+\frac{(1-a^2)^2(1+|a|\,r^{k})}{(1+|a|)(1-r^{k})}\sum_{s=1}^{\infty}a^{2s-2}r^{2sk}\\
&=&1-\frac{(1-a^2)(1-r^{2m})}{(1+ar^m)^2}+\frac{(1-a^2)r^m}{\bigl(1+ar^m\bigr)^2}+(1-a^2) \frac{a\, r^{2k}}{1-a r^{k}}+\frac{(1-a^2)^2(1+|a|\,r^{k})}{(1+|a|)(1-r^{k})}\frac{r^{2k}}{1-a^{2} r^{2k}}\\
&=&1+(1-a^2) G_2(a,r),\eeas
where 
\[G_2(a,r)=-\frac{1-r^{2m}}{(1+ar^m)^2}+\frac{r^m}{\bigl(1+ar^m\bigr)^2}+\frac{a\, r^{2k}}{1-a r^{k}}+\frac{(1-a^2)(1+|a|\,r^{k})}{(1+|a|)(1-r^{k})}\frac{r^{2k}}{1-a^{2} r^{2k}}.\]
Let $a\rightarrow 1$. It follows that 

\beas \lim_{a\rightarrow 1}G_2(a,r)=-\frac{1-r^{2m}-r^m}{(1+r^m)^2}+\frac{r^{2k}}{1- r^{k}}>0,\;\text{for}\;r>R_{2,m,k}.\eeas
Hence $\mathcal{B}(f_a(z))>1$ for $r>R_{2,m,k}$.
 
\end{proof}

\section{{\bf Refined Bohr Inequalities for Holomorphic Mappings $F : B_X \to \ol{\mathbb{U}}^n, n\geq 2$  in Banach Spaces.}}

Recently, Jia et al.~\cite{JLP}, proved several Bohr-type inequalities. We first recall some of these results.
\begin{theo3.A} \cite[Theorem 11]{JLP} Let $F \colon B_X \to \mathbb{U}$ be a holomorphic function defined by
\[
F(z)=\sum_{j=0}^{\infty} P_j(z), \qquad z \in B_X,
\]
 such that  $ a := F(0)$. Then, for
$m,k \in \mathbb{N}$ and $p \in (0,2]$, we have
\begin{equation*} |a|^{p}+ \sum_{j=1}^{\infty} \bigl|P_{qj}(\mu_k(z))\bigr|+ \left( \frac{1}{1+|a|}+ \frac{r^{qk}}{1-r^{qk}} \right)\sum_{j=1}^{\infty} \bigl|P_{qj}(\mu_k(z))\bigr|^{2}\le 1,
\end{equation*}
for $\|z\| = r \le R_1$, where $R_1 := R^{(p)}_{k,m}$  is the unique root in $(0,1)$ of the equation
\begin{equation*}\frac{r^{qk}}{1-r^{qk}} - \frac{p}{2}=0.\end{equation*}
 The number $R_1$ is best possible.
\end{theo3.A}

It is natural to pose the following problem.

\begin{que3.1}
Do analogues of Theorems 1.A and 3.A remain valid for holomorphic
mappings $F \colon B_X \to \overline{\mathbb{U}^{n}}, \; n \ge 2,$
under appropriate assumptions on the initial coefficient vector $a = F(0)$ and with
suitable multidimensional versions of the Bohr radius ?
\end{que3.1}
 
The following are key lemma of this paper and will be used to prove the main results.
\begin{lem}\label{L3}
Let $B_{X}$ be the unit ball of a complex Banach space $X$. Let $a=(a_{1},\ldots,a_{n})\in \ol{\mathbb{U}^{n}}$ and let $F$ be a holomorphic mapping from $B_{X}$ to $\ol{\mathbb{U}^{n}}$ with
\[
F(z)=a+\sum_{s=1}^{\infty}P_{s}(z), \qquad z\in B_{X},q\in \mathbb{N}
\]
where $P_{n}(z)=\frac{1}{n!}D^{n}F(0)(z^{n})$ and   $\mu_k: B_{X}\to B_{X}$ are  Schwarz mappings having $z=0$ as a zero of order $k$. Assume that $|a_{j}|=\|a\|_{\infty}$ for all $j\in\{1,2,\ldots,n\}$.
Then, for each $k\in\mathbb{N}$, the inequality
\begin{equation}
\sum_{j=1}^{\infty} \|P_{qj}(\mu_k(z))\|_{\infty}
+\left(\frac{1}{1+\|a\|_{\infty}}+\frac{r^{qk}}{1-r^{qk}}\right)
\sum_{j=1}^{\infty} \|P_{qj}(\mu_k(z))\|_{\infty}^2
\le (1-\|a\|_{\infty}^2)\frac{r^{qk}}{1-r^{qk}}
\label{eq:lemma3}
\end{equation}
holds for $\|z\|=r\in[0,1)$.
\end{lem}

\begin{proof}
Let $z\in B_X\setminus\{0\}$ be fixed and define $z_0=\frac{z}{\|z\|_X}\in \partial B_X.$ Also let $f_{j}(\zeta)=F_{j}(\zeta z_{0}), \zeta\in\mathbb{\ol U}$ for all $j=1,2,\ldots,n$, 
where $F=(F_{1},\ldots,F_{n})$. Then $f_j$ are the holomorphic functions from  $\mathbb{U}$ to $\overline{\mathbb{U}}$ and
\[
f_{j}(\zeta)=a_{j}+\sum_{s=1}^{\infty}(P_{s})_{j}(z_{0})\zeta^{s},\qquad \zeta\in\mathbb{U}.
\]
Therefore, for each fixed $j$, we have (see e.g. \cite[p.\,35]{GK2003})
\bea\label{l0}
|(P_{s})_{j}(z_{0})| \le 1-|a_{j}|^{2}=1-\|a\|_{\infty}^{2}.
\eea
and  applying Lemma \ref{L2}  for each $j\in \{1,2,\ldots,n\}$, we get
\bea\label{l1} &&\sum_{s=1}^{\infty}|(P_{s})_{j}(z_{0})|\,|\zeta|^{s}
+\left(\frac{1}{1+|a_j|}+\frac{|\zeta|}{1-|\zeta|}\right)
\sum_{s=1}^{\infty} |(P_{s})_{j}(z_{0})|^2 |\zeta|^{2s}\nonumber\\
& \le & (1-|a_j|^2)\frac{|\zeta|}{1-|\zeta|}=(1-\|a\|_{\infty}^2)\frac{|\zeta|}{1-|\zeta|}.
\eea
Since (\ref{l1}) are hold for all $j\in \{1,2,\ldots,n\}$, then we deduce that 
\bea\label{l2} &&\sum_{s=1}^{\infty}\max_{j\in\{1,2,\ldots,n\}}|(P_{s})_{j}(z_{0})|\,|\zeta|^{s}
+\left(\frac{1}{1+|a_j|}+\frac{|\zeta|^q}{1-|\zeta|^q}\right)
\sum_{s=1}^{\infty} \max_{j\in\{1,2,\ldots,n\}} |(P_{s})_{j}(z_{0})|^2 |\zeta|^{2s}\nonumber\\
& \le &(1-\|a\|_{\infty}^2)\frac{|\zeta|}{1-|\zeta|}.
\eea
Form (\ref{l0}), we get 
\[\|P_{s}(z_0)\|_{\infty}:=\max_{j\in \{1,2,\ldots,n\}} |(P_{qs})_{j}(z_{0})|\leq 1-\|a\|_{\infty}^{2}.\]
Let $\zeta=\|z\|_{X}^q=r^q<1$. Then, by the above equation and (\ref{l2}), it follows that
\bea\label{hhh1} \sum_{s=1}^{\infty}\|P_{s}(z)\|_{\infty} +\left(\frac{1}{1+\|a\|_{\infty}}+\frac{\|z\|_X^q}{1-\|z\|_X^q}\right)
\sum_{s=1}^{\infty}  \|P_{s}(z)\|_{\infty}^{2}
\le (1-\|a\|_{\infty}^2)\frac{\|z\|_X^q}{1-\|z\|_X^q}.\eea
Note that $s(x)=\frac{x}{1-x}$. Then $s'(x)=\frac{1}{(1-x)^2}>0,$ so $s(x)$ is increasing on $(0,1)$, since $r^{qk}\leq r^q$ for $k\geq 1$, it follows that 
\[s(r^{qk})\leq s(r^q).\]
Hence 
\[\frac{r^{qk}}{1-r^{qk}}\leq \frac{r^q}{1-r^q} =\frac{\|z\|_X^q}{1-\|z\|_X^q}.\]
That is 
\[\frac{1}{1+\|a\|_{\infty}}+\frac{r^{qk}}{1-r^{qk}}\leq \frac{1}{1+\|a\|_{\infty}}+\frac{\|z\|_X^q}{1-\|z\|_X^q}\]
Also 
\[\sum_{s=1}^{\infty}\|P_{qs}(z)\|_{\infty}\leq \sum_{s=1}^{\infty}\|P_{s}(z)\|_{\infty}\;\text{and}\; \sum_{s=1}^{\infty}\|P_{qs}(z)\|_{\infty}^2\leq \sum_{s=1}^{\infty}\|P_{s}(z)\|_{\infty}^2.\]
Therefore applying above inequalities, we get
\bea\label{ne3} &&\sum_{s=1}^{\infty}\|P_{qs}(z)\|_{\infty}+\left(\frac{1}{1+\|a\|_{\infty}}+\frac{r^{qk}}{1-r^{qk}}\right)\sum_{s=1}^{\infty}\|P_{qs}(z)\|_{\infty}^2\nonumber\\
&\leq & \sum_{s=1}^{\infty}\|P_{qs}(z)\|_{\infty}+\left(\frac{1}{1+\|a\|_{\infty}}+\frac{\|z\|_X^q}{1-\|z\|_X^q}\right)\sum_{s=1}^{\infty}\|P_{qs}(z)\|_{\infty}^2\nonumber\\
&\leq & \sum_{s=1}^{\infty}\|P_{s}(z)\|_{\infty}+\left(\frac{1}{1+\|a\|_{\infty}}+\frac{\|z\|_X^q}{1-\|z\|_X^q}\right)\sum_{s=1}^{\infty}\|P_{s}(z)\|_{\infty}^2.\eea
Hence from (\ref{hhh1}) and (\ref{ne3}), we get
\bea\label{ne4} \sum_{s=1}^{\infty}\|P_{qs}(z)\|_{\infty}+\left(\frac{1}{1+\|a\|_{\infty}}+\frac{r^{qk}}{1-r^{qk}}\right)\sum_{s=1}^{\infty}\|P_{qs}(z)\|_{\infty}^2\leq (1-\|a\|_{\infty}^2)\frac{\|z\|^q}{1-\|z\|^q}.\eea

In view of  (\ref{eq:Schwarz}) and (\ref{ne4}), we obtain  
\beas &&\sum_{s=1}^{\infty}\|P_{qs}(\mu_k(z))\|_{\infty} +\left(\frac{1}{1+\|a\|_{\infty}}+\frac{r^{qk}}{1-r^{qk}}\right)
\sum_{s=1}^{\infty}  \|P_{qs}(\mu_k(z))\|_{\infty}^{2} \\
&\leq &(1-\|a\|_{\infty}^2)\frac{\|\mu_k(z)\|^q}{1-\|\mu_k(z)\|^q}\;\left( \text{as $\frac{x}{1-x}$ is increasing $(0,1)$}\right)\\
& \leq & (1-\|a\|_{\infty}^2)\frac{r^{qk}}{1-r^{qk}}.\eeas
\end{proof}
\begin{lem}\label{L4}\cite{CLP2023}
Let $m \in \mathbb{N}$ and $p \in (0,2]$. For $a \in [0,1]$, consider
\[
D_{p,m}(a)
= \left( \left( \frac{a + r^{m}}{1 + a r^{m}} \right)^{p} - 1 \right)\varphi_{0}(r)
+ (1-a^{2})\,N(r),
\]
where $\varphi_{0}(r)$ and $N(r)$ are nonnegative continuous functions defined on $[0,1)$.
Suppose further that
\[
\Phi_{p,m}(r)
= p\left( \frac{1 - r^{m}}{1 + r^{m}} \right)\varphi_{0}(r)
- 2N(r),
\]
and let $R := R(m,p)$ be the minimal positive root in $(0,1)$ of the equation
$\Phi_{p,m}(r) = 0.$
If $\Phi_{p,m}(r) \ge 0$ for $0 \le r \le R$, then
$D_{p,m}(a) \le 0 \quad \text{for } 0 \le r \le R.$
\end{lem}

\section*{{\bf Main results and their proofs.}}
In this section, we resolve Question 3.1 and present the following main results.
\begin{theo}\label{T2.1}
Let $B_{X}$ be the unit ball of a complex Banach space $X$. Let $a=(a_{1},\ldots,a_{n})\in \ol{\mathbb{U}^{n}}$ and let $F$ be a holomorphic mapping from $B_{X}$ to $\ol{\mathbb{U}^{n}}$ with
\[
F(z)=a+\sum_{s=1}^{\infty}P_{s}(z), \qquad z\in B_{X},
\]
where $P_{s}(z)=\frac{1}{s!}D^{s}F(0)(z^{s}).$ Assume that $|a_{j}|=\|a\|_{\infty}$ for all $j\in\{1,2,\ldots,n\}$. Then for $m,k\in\mathbb{N}$ and $p\in (0,2]$, we have
\begin{equation}\label{e2.1}
\|F( \mu_m(z)) \|_{\infty}^p +  \sum_{s=1}^{\infty} \| P_{2s}(\mu_k(z)) \|_{\infty}  \leq 1.
\end{equation}
for $\|z\|=r\le R_{1,k,m,p}$, where $\mu_k,\mu_m:B_{X}\to B_{X}$ are Schwarz mappings having $z=0$ as a zero of order $k,m$ respectively and $R_{1,k,m,p}$ is the unique root in $(0,1)$ of the equation
\begin{equation}\label{e2.2} p\left(\frac{1-r^{m}}{1+r^{m}}\right)- \frac{2r^{2k}}{1-r^{2k}}=0.\end{equation}
The constant $R_{1,k,m,p}$ cannot be improved.
\end{theo}
\begin{rem} In the special case where $B_X = B_Y = \mathbb{U}$, $\mu_m(z)=\mu_k(z)=z$, and $p=1$, 
Theorem~\ref{T2.1} reduces to Theorem~1.A.

\end{rem}
\begin{table}[H]
\centering
\begin{minipage}[b]{0.45\textwidth}
\centering
\begin{tabular}{|c|c|c|c|}
\hline
\textbf{k} & \textbf{m} & \textbf{p} & \textbf{$R_{1,k,m,p}$} \\ 
\hline
1 & 1 & 1& 0.414200 \\ 
\hline
1 & 1 & 2& 0.500000 \\ 
\hline
2 & 2 & 3 & 0.742000 \\ 
\hline
3 & 2 & 4& 0.813300 \\ 
\hline
\end{tabular}
\caption{$R_{1,k,m,p}$ is the unique root of equation (\ref{e2.2}) in $(0,1)$.}
\label{tab3.3}
\end{minipage}%
\hfill
\begin{minipage}[b]{0.55\textwidth}
\centering
\begin{figure}[H]
\centering
\includegraphics[width=\textwidth]{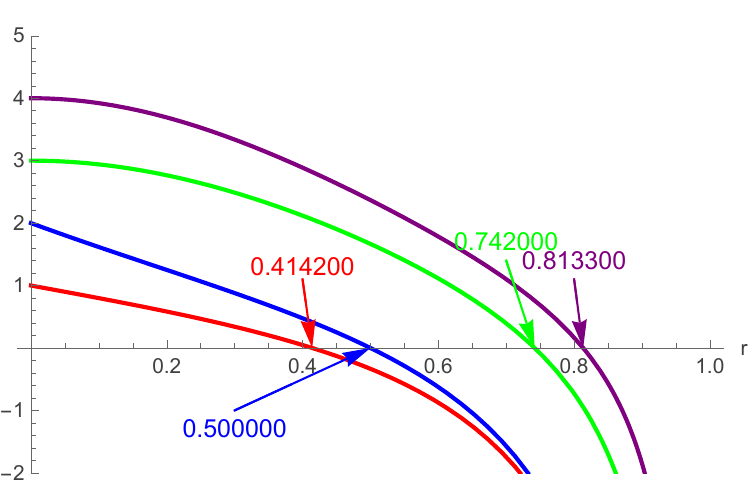}
\caption{The graphs exhibit the locations of the roots $R_{1,k,m,p}$ in $(0,1)$ for different values of $k,m,p$.}
\label{f33}
\end{figure}
\end{minipage}
\end{table}

\begin{proof}[{\bf Proof of Theorem \ref{T2.1}}]
Let $z\in B_X\setminus\{0\}$ be fixed and define $z_0=\frac{z}{\|z\|_X}\in \partial B_X$.  
We consider the holomorphic functions $f_j : \mathbb{U} \to \overline{\mathbb{U}}$ defined by
\[
f_j(\zeta) = F_j(\zeta z_0), \qquad \zeta \in \mathbb{U}
\]
where $F=(F_1,\ldots,F_n)$. Then 
\begin{equation}\label{eq2.1}
f_j(\zeta)=a_j+\sum_{s=1}^{\infty}(P_s)_j(z_0)\zeta^s,
\qquad \zeta\in\mathbb{U}.
\end{equation}
Hence, for each fixed $j$ (see, e.g., \cite[p.~35]{GK2003}),
$|(P_s)_j(z_0)|\leq 1-|a_j|^2=1-\|a\|_{\infty}^2.$ Note that \[\|(P_s(z_0)\|_{\infty}=\max\{|(P_s)_{1}(z_0)|,|(P_s)_{2}(z_0)| \cdots |(P_s)_{n}(z_0)| \} \leq 1-\|a\|_{\infty}^2\]
Consequently
\begin{equation}\label{eq2.1}
\|P_{2s}(z_0)\|_{\infty}\le 1-\|a\|_{\infty}^2 .
\end{equation}
Now 
\begin{align*}
\sum_{s=1}^{\infty} \left| (P_{2s})_j(z_0) \right| |\zeta|^{2s}
&\le \sum_{s=1}^{\infty} \left(1-\|a\|_{\infty}^2\right) |\zeta|^{2s}\\
&= \left(1-\|a\|_{\infty}^2\right) \frac{|\zeta|^{2}}{1-|\zeta|^{2}}  \qquad \forall j=1,2,\ldots,n.
\end{align*}
Note that \[ \|P_{2s}(z_0)\|_{\infty}=\max \{\left| (P_{2s})_1(z_0) \right|,\left| (P_{2s})_2(z_0) \right|, \cdots, \left| (P_{2s})_n(z_0) \right|\}.\]
Consequently we get 
\begin{equation}\label{pd1}
\sum_{s=1}^{\infty} {\| P_{2s}(z_0) \|}_{\infty} |\zeta|^{2s} \leq \left(1-\|a\|_{\infty}^2\right) \frac{|\zeta|^{2}}{1-|\zeta|^{2}}
\end{equation}
Setting $\zeta=\|z\|_X=r<1$, from (\ref{pd1}) we get
\begin{align*}
\sum_{s=1}^{\infty} \| P_{2s}(z_0 \zeta) \|_{\infty}
&= \sum_{s=1}^{\infty} \left\| P_{2s}\!\left(\frac{z}{\|z\|_{X}}\,\|z\|_X\right) \right\|_{\infty} \\
&\le \left(1-\|a\|_{\infty}^2\right)\frac{\|z\|^{2}}{1-\|z\|^{2}},
\end{align*}
i.e
\begin{align}\label{pd2}
\sum_{s=1}^{\infty} \| P_{2s}(z) \|_{\infty} \leq (1-\|a\|_{\infty}^2)\frac{\|z\|_X^2}{1-\|z\|_X^2}.
\end{align}
Note that using Lemma \ref{L1} and (\ref{eq:Schwarz}), we get
\[\|F(\mu_m(z))\|_{\infty} \le \left(\frac{\|a\|_{\infty}+\|\mu_m(z))\|_{X}}{1+\|a\|_{\infty}\|\mu_m(z))\|_{X}}.\right)\]
Let 
\[
\phi(t)=\frac{\|a\|_{\infty}+t}{1+\|a\|_{\infty} t}, \qquad t=\|\mu_m(z)\|_X, t_1=\|z\|_X^{m}.
\]
Then $\phi'(t)=\frac{1-\|a\|_{\infty}^{2}}{(1+\|a\|_{\infty}t)^{2}}>0$ and hence \(\phi\) is increasing on \([0,1)\). Therefore, for \(t\le t_1\), we get
$\phi(t)\le \phi(t_1).$ Consequently, we obtain
\begin{equation}\label{pd5}
\left(
\frac{\|a\|_{\infty}+\|\mu_m(z)\|_{X}}
{1+\|a\|_{\infty}\,\|\mu_m(z)\|_{X}}
\right)^p
\le
\left(
\frac{\|a\|_{\infty}+\|z\|_{X}^{m}}
{1+\|a\|_{\infty}\,\|z\|_{X}^{m}}
\right)^p .
\end{equation}
Let $x=\|a\|_{\infty}\in[0,1]$. If $x=1$, then $P_s(z_0)=0$ for all $s\ge1$, and
\eqref{e2.1} holds trivially. Hence, we assume $x\in[0,1)$.
Now using Eq. (\ref{pd2}) and (\ref{pd5}) we get

\begin{align*}\label{pd6}
\|F( \mu_m(z)) \|_{\infty}^p + \sum_{s=1}^{\infty} \| P_{2s}(\mu_k(z)) \|_{\infty}  
&\leq \left(
\frac{x+r^{m}}
{1+x r^{m}}
\right)^p + (1-x^2)\frac{r^{2k}}{1-r^{2k}} \\
&= 1 + D_{3}(x,r)
\end{align*}
 
where

\[
D_{3}(x,r)
= \left(\frac{x+r^m}{1+xr^m}\right)^p - 1
+ \left(1-x^2\right)\frac{r^{2k}}{1-r^{2k}},
\qquad x=\|a\|_{\infty}.
\]
By Lemma \ref{L4}, we deduce that $D_{3}(x,r)\leq 0$, for $0\leq r\leq R_{1,k,m,p}$,
where $R_{1,k,m,p}$ is the minimal positive root in $(0,1)$ of the equation $p\left(\frac{1-r^{m}}{1+r^{m}} \right)- \frac{2r^{2k}}{1-r^{2k}}=0.$\par

\medskip

{\bf Optimality of the Radius $R_{1,k,m,p}$.} 
We next verify that the radius $R_{1,k,m,p}$ cannot be improved.

Fix $z_0\in\partial B_X$ and $u=(u_1,u_2\ldots,u_n)\in\partial\mathbb{U}_n$ with
$|u_1|=\cdots=|u_n|=1$.
For $b\in(0,1)$, define
\begin{equation}\label{F1}
F_1(z)=f(l_{z_0}(z))u,\qquad z\in B_X,
\end{equation}
where
\[
f(\zeta)=\frac{b-\zeta}{1-b\zeta},\qquad \zeta\in\mathbb{U},
\]
and $l_{z_0}\in T(z_0)$.
Let $\mu_m(z)=-l_{z_0}(z)^{\,m-1}z$ for $m\geq 1$.
Then, for $r\in(0,1)$,
\[
F_1(rz_0)=\frac{b-r}{1-br}\,u
= a+\sum_{n=1}^{\infty}P_n(rz_0),
\]
where $a=bu$ and $P_n(z_0)=-b^{\,n-1}(1-b^2)u, n\ge 1.$
Consequently,
\[
P_n(\mu_k(rz_0))
=-r^{kn}b^{\,n-1}(1-b^2)u.
\]
Thus for $z=rz_0$  we have
\beas \|F( \mu_m(rz_0)) \|_{\infty}^p + \sum_{s=1}^{\infty} \| P_{2s}(\mu_k(rz_0)) \|_{\infty} &=&\left(\frac{b+r^m}{1+br^m}\right)^p+\frac{b(1-b^2) r^{2k}}{1-b^2 r^{2k}}\\
&=&1+(1-b)G_3(b,r),\eeas
where $G_3(b,r)=\frac{1}{1-b}\left( \left(\frac{b+r^m}{1+br^m}\right)^p-1\right)+\frac{(1+b) r^{2k}}{1-b r^{2k}}$. Let $r\in (R_{1,k,m,p},1)$ be arbitrary fixed. Since $R_{1,k,m,p}$ satisfies Eq (\ref{e2.2}), it follows that 
\[-p\left(\frac{1-r^{m}}{1+r^{m}} \right)+ \frac{2r^{2k}}{1-r^{2k}}>0\]
for $r\in (R_{1,k,m,p},1)$. Therefore 

\[\lim_{b\rightarrow 1^-}G_3(b,r)=-p\left(\frac{1-r^{m}}{1+r^{m}} \right)+ \frac{2r^{k}}{1-r^{k}}>0.\]
This implies $\|F( \mu_m(rz_0)) \|_{\infty}^p + \big( \sum_{s=1}^{\infty} \| P_{2s}(\mu_k(rz_0)) \|_{\infty} \big)>1 $  for $r\in (R_{1,k,m,p},1)$.
Thus $R_{1,k,m,p}$ is optimal. 

\end{proof}
The following example shows that the condition
$|a_j|=\|a\|_{\infty}$ for all $j\in\{1,2,\ldots,n\}$
cannot, in general, be relaxed.
\begin{exm}\label{ex1}
Let $n=2$, $z_{0}\in \partial B_{X}$ and $l_{z_{0}}\in T(z_{0})$. 
Let $\alpha,\beta\in \mathbb{R}$ satisfy $\frac{1}{\sqrt{3}}<\alpha<\beta<1.$
For $m\ge 1$, define $\mu_m(z)=-l_{z_{0}}(z)^{\,m-1}z .$
Define the holomorphic mapping $F:B_{X}\to \overline{\mathbb{U}}^{2}$ by
\bea\label{E1}
F(z)=\left(
\frac{\alpha-l_{z_{0}}(z)}{1-\alpha l_{z_{0}}(z)},
\frac{\beta-l_{z_{0}}(z)}{1-\beta l_{z_{0}}(z)}
\right), \qquad z\in B_{X}.
\eea

For $s\ge 1$, $P_{2s}(rz_{0})
=-\bigl((1-\alpha^{2})\alpha^{2s-1}r^{2s},\,
(1-\beta^{2})\beta^{2s-1}r^{2s}\bigr).$
Consequently, $P_{2s}(\mu_{k}(rz_{0}))
=-\bigl((1-\alpha^{2})\alpha^{2s-1}r^{2ks},\,
(1-\beta^{2})\beta^{2s-1}r^{2ks}\bigr),$
and hence $
\|P_{2s}(\mu_{k}(rz_{0}))\|_{\infty}
=(1-\alpha^{2})\alpha r^{2k},  s=1$
while $\|P_{2s}(\mu_{k}(rz_{0}))\|_{\infty}
\ge (1-\beta^{2})\beta^{2s-1}r^{2ks},  s\ge 2.$\par

Hence,
\begin{align*}
&\|F(\mu_m(rz_0))\|_{\infty}^p
+\sum_{s=1}^{\infty}\|P_{2s}(\mu_{k}(rz_{0}))\|_{\infty}
  \\
&= \left(\frac{\beta+r^m}{1+\beta r^m}\right)^p+(1-\alpha^{2})\alpha r^{2k}+\sum_{s=2}^{\infty}\|P_{2s}(\mu_{k}(rz_{0}))\|_{\infty}\\
&\geq \left(\frac{\beta+r^m}{1+\beta r^m}\right)^p+(1-\alpha^{2})\alpha r^{2k}+\sum_{s=2}^{\infty}(1-\beta^{2})\beta^{2s-1}r^{2ks}\\
&\geq \left(\frac{\beta+r^m}{1+\beta r^m}\right)^p+(1-\alpha^{2})\alpha r^{2k}+\frac{(1-\beta^{2})\,\beta^{3} r^{4k}}{1-\beta^{2} r^{2k}}.
\end{align*}
Now by taking $\beta\rightarrow 1^-$, we get $\|F(\mu_m(rz_0))\|_{\infty}^p
+\sum_{s=1}^{\infty}\|P_{2s}(\mu_{k}(rz_{0}))\|_{\infty}\geq 1+(1-\alpha^{2})\alpha r^{2k},$
which implies there is no-Bohr radius in this case.
\end{exm}

\begin{theo}\label{T1.2}
Let $B_{X}$ be the unit ball of a complex Banach space $X$. Let $a=(a_{1},\ldots,a_{n})\in \ol{\mathbb{U}^{n}}$ and let $F$ be a holomorphic mapping from $B_{X}$ to $\ol{\mathbb{U}^{n}}$ with $F(z)=a+\sum_{s=1}^{\infty}P_{s}(z),  z\in B_{X},$
where $P_{n}(z)=\frac{1}{n!}D^{n}F(0)(z^{n}), \;\forall n\geq 1$. Assume that $|a_{j}|=\|a\|_{\infty}$ for all $j\in\{1,2,\ldots,n\}$. Then for $m,k,q\in\mathbb{N}$ and $p\in (0,2]$, we have
\begin{equation}\label{2.1}
\|a\|_{\infty}^{p} + \sum_{s=1}^{\infty}\|P_{qs}(\mu_k(z))\|_{\infty}+ \left(\frac{1}{1+\|a\|_{\infty}}+\frac{r^{qk}}{1-r^{qk}}\right)\sum_{s=1}^{\infty}\|P_{qs}(\mu_k(z))\|_{\infty}^2\le 1.
\end{equation}
for $\|z\|=r\le R_{1,k,p,q}$, where $\mu_k:B_{X}\to B_{X}$ are Schwarz mappings having $z=0$ as a zero of order $k$ and $R_{1,k,p,q}$ is the unique root in $(0,1)$ of the equation
\begin{equation}\label{2.2}\frac{r^{qk}}{1-r^{qk}}  - \frac{p}{2}=0.\end{equation}
The constant $R_{1,k,p,q}$ cannot be improved.
\end{theo}
\begin{rem} In the special case where $B_Y=  \mathbb{U}$, 
Theorem~\ref{T2.2} reduces to Theorem~3.A.

\end{rem}
\begin{table}[H]
\centering
\begin{minipage}[b]{0.45\textwidth}
\centering
\begin{tabular}{|c|c|c|c|}
\hline
\textbf{k} & \textbf{p} & \textbf{q} & \textbf{$R_{1,k,p,q}$} \\ 
\hline
1 & 1 & 1& 0.333300 \\ 
\hline
2 & 3 & 3& 0.918400 \\ 
\hline
3 & 5 & 4 & 0.972400 \\ 
\hline
4 & 7 & 5& 0.987500 \\ 
\hline
\end{tabular}
\caption{$R_{1,k,p,q}$ is the unique root of equation (\ref{2.2}) in $(0,1)$.}
\label{tab3.4}
\end{minipage}%
\hfill
\begin{minipage}[b]{0.55\textwidth}
\centering
\begin{figure}[H]
\centering
\includegraphics[width=\textwidth]{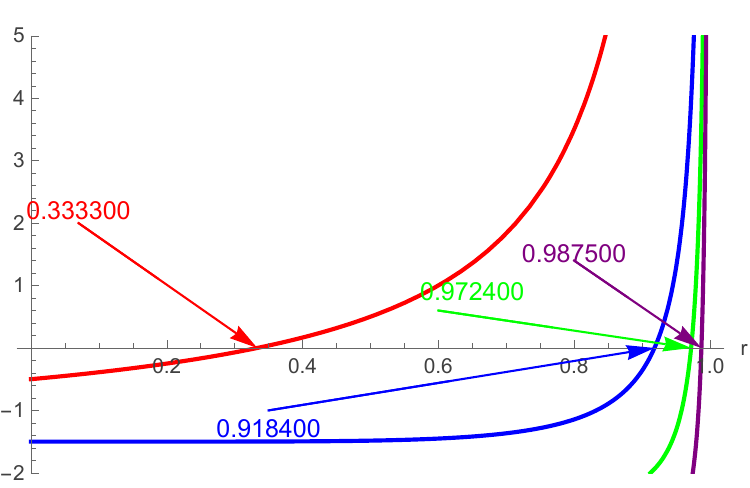}
\caption{The graphs exhibit the locations of the roots $R_{1,k,p,q}$ in $(0,1)$ for different values of $k,p,q.$}
\label{f34}
\end{figure}
\end{minipage}
\end{table}

\begin{proof}[{\bf Proof of Theorem \ref{T1.2}}]
Let $z\in B_X\setminus\{0\}$ be fixed and define $z_0=\frac{z}{\|z\|_X}\in \partial B_X.$
We consider the holomorphic functions $f_j : \mathbb{U} \to \overline{\mathbb{U}}$ defined by
\[
f_j(\zeta) = F_j(\zeta z_0), \qquad \zeta \in \mathbb{U}
\]
where $F=(F_1,\ldots,F_n)$. Then 
\begin{equation}\label{eq:fexp}
f_j(\zeta)=a_j+\sum_{s=1}^{\infty}(P_s)_j(z_0)\zeta^s,
\qquad \zeta\in\mathbb{U}.
\end{equation}
Hence, for each fixed $j$ (see, e.g., \cite[p.~35]{GK2003}),
$|(P_s)_j(z_0)|\le 1-|a_j|^2=1-\|a\|_{\infty}^2.$
This yields
\begin{equation}\label{eq:Ps}
\|P_s(z_0)\|_{\infty}\le 1-\|a\|_{\infty}^2 .
\end{equation}

Let $x=\|a\|_{\infty}\in[0,1]$. If $x=1$, then $P_s(z_0)=0$ for all $s\ge1$, and
\eqref{2.1} holds trivially. Hence, we assume $x\in[0,1)$.
Now applying Lemma~\ref{L3}, we obtain
\begin{align}
\mathcal{A}_f(z):=&\|a\|_{\infty}^{p}
+\sum_{j=1}^{\infty}\|P_{qj}(\mu_k(rz_0))\|_{\infty}
+\left(\frac{1}{1+\|a\|_{\infty}}+\frac{r^{qk}}{1-r^{qk}}\right)
\sum_{j=1}^{\infty}\|P_{qj}(\mu_k(rz_0))\|_{\infty}^2\nonumber\\ 
&\le x^p+(1-x^2)\frac{r^{qk}}{1-r^{qk}} \nonumber\\
&=1+(1-x^2)D_4(x,r),
\label{eq:G1}
\end{align}
where
\[
D_4(x,r)
=\frac{r^{qk}}{1-r^{qk}}
-\frac{1-x^p}{1-x^2}.
\]

Suppose that $A(x)=\frac{1-x^p}{1-x^2}$. Since $A(x)$ is decreasing on $[0,1)$ for $p\in(0,2]$, we have
\[
A(x)\ge \lim_{x\to1^-}A(x)=\frac{p}{2}.
\]
Consequently,
\[
D_4(x,r)\le D_4(1,r)
=\frac{r^{qk}}{1-r^{qk}}
-\frac{p}{2}.
\]
Thus $\mathcal{A}_f(z)\le1$ whenever $D_4(1,r)\le0$, which holds for
$r\le R_{1,k,p,q}$, where $R_{1,k,p,q}$ is the smallest positive root in $(0,1)$ of
\begin{equation}\label{eq:R1}
\frac{r^{qk}}{1-r^{qk}}
-\frac{p}{2}=0.
\end{equation}
Let $\Phi(r)=\frac{r^{qk}}{1-r^{qk}}-\frac{p}{2}$. Clearly $\Phi'(r)>0$, this shows that $\Phi(r)$ is increasing in $(0,1)$. Note that 
\[\Phi(0)=-\frac{p}{2}\;\;\text{and}\;\;\lim_{r\rightarrow 1^-}\Phi(r)=+\infty.\]
Therefore $R_{1,k,p,q}$ is the unique  root in $(0,1)$ of $\Phi(r)=0$.

\medskip

{\bf Optimality of the Radius $R_{1,k,p,q}$:} 
We now show that the radius $R_{1,k,p,q}$ is sharp.
Fix $z_0\in\partial B_X$ and $u=(u_1,\ldots,u_n)\in\partial\mathbb{U}_n$ with
$|u_1|=\cdots=|u_n|=1$.
For $b\in(0,1)$, define
\begin{equation}\label{F1}
F_1(z)=f(l_{z_0}(z))u,\qquad z\in B_X,
\end{equation}
where
\[
f(\zeta)=\frac{b-\zeta}{1-b\zeta},\qquad \zeta\in\mathbb{U},
\]
and $l_{z_0}\in T(z_0)$.
Let $\mu_m(z)=-l_{z_0}(z)^{\,m-1}z$ for $m\geq 1$.
Then, for $r\in(0,1)$,
\[
F_1(rz_0)=\frac{b-r}{1-br}\,u
= a+\sum_{n=1}^{\infty}P_n(rz_0),
\]
where $a=bu$ and $P_n(z_0)=-b^{\,n-1}(1-b^2)u, n\ge 1.$
Consequently,
\[
P_n(\mu_k(rz_0))
=-r^{kn}b^{\,n-1}(1-b^2)u.
\]

Thus we have 
\begin{align*}
\mathcal{A}_f(z):=&\|a\|_{\infty}^{p}
+\sum_{j=1}^{\infty}\|P_{qj}(\mu_k(rz_0))\|_{\infty}
+\left(\frac{1}{1+\|a\|_{\infty}}+\frac{r^k}{1-r^k}\right)
\sum_{j=1}^{\infty}\|P_{qj}(\mu_k(rz_0))\|_{\infty}^2 \\
=&1+(1-b^2)G_4(b,r),
\end{align*}
where
\[
G_4(b,r)
=\frac{ b^{q-1} r^{qk}}{1-b^{q} r^{qk}}
+\left(\frac{1}{1+b}+\frac{r^{qk}}{1-r^{qk}}\right)
\frac{(1-b^2)}{b^{2}}\,
\frac{b^{2q} r^{2qk}}{1-b^{2q} r^{2qk}}
-\frac{1-b^p}{1-b^2}.
\]
Letting $b\to1^-$, we obtain
\[
\lim_{b\to1^-}G_4(b,r)
=\frac{r^{qk}}{1-r^{qk}}
-\frac{p}{2}>0,
\qquad r>R_{1,k,p,q}.
\]
Hence, for $r>R_{1,k,p,q}$ and $b$ sufficiently close to $1$, we have
$\mathcal{A}_f(z)>1$.
This proves that the radius $R_{1,k,p,q}$ is best possible.
\end{proof}
The following example shows that the assumption $|a_j|=\|a\|_{\infty}$ for all $j\in\{1,2,\ldots,n\}$
cannot be omitted in the above result.

\begin{exm}\label{ex1}
Let $n=2$, $z_{0}\in \partial B_{X}$ and $l_{z_{0}}\in T(z_{0})$. 
Let $\alpha,\beta\in \mathbb{R}$ satisfy $\frac{1}{\sqrt{3}}<\alpha<\beta<1.$
For $m\ge 1$, define $\mu_m(z)=-l_{z_{0}}(z)^{\,m-1}z .$
Define the holomorphic mapping $F:B_{X}\to \overline{\mathbb{U}}^{2}$ by
\bea\label{E1}
F(z)=\left(
\frac{\alpha-l_{z_{0}}(z)}{1-\alpha l_{z_{0}}(z)},
\frac{\beta-l_{z_{0}}(z)}{1-\beta l_{z_{0}}(z)}
\right), \qquad z\in B_{X}.
\eea
Then $a=P_{0}(z)=F(0)=(\alpha,\beta),$
and for $s\ge 1$, $P_{s}(rz_{0})
=-\bigl((1-\alpha^{2})\alpha^{s-1}r^{s},\,
(1-\beta^{2})\beta^{s-1}r^{s}\bigr).$
Consequently, $P_{s}(\mu_{k}(rz_{0}))
=-\bigl((1-\alpha^{2})\alpha^{s-1}r^{ks},\,
(1-\beta^{2})\beta^{s-1}r^{ks}\bigr),$
and hence $
\|P_{s}(\mu_{k}(rz_{0}))\|_{\infty}
=(1-\alpha^{2})\alpha^{s-1}r^{ks},  s=1,2,$
while $\|P_{s}(\mu_{k}(rz_{0}))\|_{\infty}
\ge (1-\beta^{2})\beta^{s-1}r^{ks},  s\ge 3.$\par

Hence,
\begin{align*}
\mathcal{A}_{f}(rz_{0})
&=\|a\|_{\infty}^{p}
+\sum_{j=1}^{\infty}\|P_{qj}(\mu_{k}(rz_{0}))\|_{\infty}
+\frac{1+\|a\|_{\infty}r^{qk}}{(1-r^{qk})(1+\|a\|_{\infty})}
\sum_{j=1}^{\infty}\|P_{qj}(\mu_{k}(rz_{0}))\|_{\infty}^{2}  \\
&\ge \beta^{p}
+(1-\alpha^{2})(r^{qk}+\alpha r^{2qk})
+(1-\beta^{2})\sum_{j=3}^{\infty}\beta^{j-1}r^{kqj} \\
&\quad +\frac{1+\beta r^{qk}}{(1-r^{qk})(1+\beta)}
\left((1-\alpha^{2})^{2}(r^{2qk}+\alpha^{2}r^{4qk})
+(1-\beta^{2})^{2}\sum_{j=3}^{\infty}\beta^{2qj-2}r^{2qkj}\right).
\end{align*}
Evaluating the above series, we obtain
\begin{align*}
\mathcal{A}_{f}(rz_{0})
&\ge \beta^{p}
+(1-\alpha^{2})(r^{qk}+\alpha r^{2qk})
+(1-\beta^{2})\frac{\beta^{2}r^{3qk}}{1-\beta r^{qk}} \\
&\quad +\frac{1+\beta r^{qk}}{(1-r^{qk})(1+\beta)}
\left((1-\alpha^{2})^{2}(r^{2qk}+\alpha^{2}r^{4qk})
+(1-\beta^{2})^{2}\frac{\beta^{4}r^{6qk}}{1-\beta^{2}r^{2qk}}\right).
\end{align*}

Finally, letting $\beta\to 1^{-}$, for each $r\in(0,1)$ we obtain
\[
\mathcal{A}_{f}(rz_{0})
\ge 1
+(1-\alpha^{2})(r^{qk}+\alpha r^{2qk})
+\frac{1+r^{qk}}{2(1-r^{qk})}(1-\alpha^{2})^{2}
(r^{2qk}+\alpha^{2}r^{4qk})
>1.
\]

\end{exm}

\section{{\bf Bohr radius for holomorphic mappings with values in the closure of the unit ball of a JB$^*$-triple}}

In this section, we study the Bohr radius for holomorphic mappings defined on the unit ball of a complex Banach space and taking values in \(\ol{\mathbb{B}_Y}\), where \(\mathbb{B}_Y\) denotes the unit ball of a JB$^*$-triple \(Y\). For the definition and basic properties of JB$^*$-triples, we refer the reader to, for example, \cite{Chu2012,Chu2021}. Let \(\mathrm{Aut}(\mathbb{B}_Y)\) denote the group of all biholomorphic self-maps of \(\mathbb{B}_Y\). It is well known that \(\mathrm{Aut}(\mathbb{B}_Y)\) acts transitively on \(\mathbb{B}_Y\).\par

The following result concerning the Bohr--Rogosinski radius for holomorphic mappings with values in the closure of the unit ball of a JB$^{*}$-triple and involving Schwarz mappings was first established in \cite[Theorem~3.3]{HHK2025}. It reads as follows.
\begin{theo4.A}\cite[Theorem 3.3]{HHK2025}
Let $B_X$ be the unit ball of a complex Banach space $X$ and let $B_Y$ be the unit ball of a JB$^{*}$-triple $Y$. Let $F$ be a holomorphic mapping from $B_X$ to $B_Y$ with
\[
F(z)=a+\sum_{s=1}^{\infty} P_s(z), \qquad z\in B_X,
\]
where $P_s(z)=\frac{1}{s!}D^sF(0)(z^s).$
Assume that $a\in B_Y$. Then, for $m,N\in\mathbb{N}$ and $p\in(0,2]$, we have
\[
\|F(v(z))\|^{p}
+
\sum_{s=N}^{\infty}
\frac{\|D\varphi_a(a)P_s(z)\|}
     {\|D\varphi_a(a)\|}
\le 1
\]
for $\|z\|=r\le R_{m,N}^{p}$, where $v:B_X\to B_X$ is a Schwarz mapping having $z=0$ as a zero of order $m$, $\varphi_a\in \operatorname{Aut}(B_Y)$ satisfies $\varphi_a(a)=0$, and $R_{m,N}^{p}$ is the unique root in $(0,1)$ of the equation
\[
p\left(\frac{1-r^{m}}{1+r^{m}}\right)
-\frac{2r^{N}}{1-r}
=0.
\]
Moreover, the number $R_{m,N}^{p}$ cannot be improved.
\end{theo4.A}
In view of the following lemma, we obtain the Bohr radius for holomorphic
mappings with values in the closure of the unit ball of a JB$^*$-triple by replacing
$\|P_s(z)\|$ with
\[
\frac{\|D\phi_a(a)P_s(z)\|}{\|D\phi_a(a)\|},
\]
where $\phi_a \in \mathrm{Aut}(\mathbb{B}_Y)$ satisfies $\phi_a(a)=0$.
\begin{lem}\label{lem:3.2}\cite[Lemma 3.1]{HH2024}
Let $F:\mathbb{U} \to \ol{\mathbb{B}_Y}$ be a holomorphic mapping of the form
\[
F(\zeta) = a + \sum_{s=1}^{\infty} P_s(\zeta), \qquad \zeta \in \mathbb{U},
\]
where $P_s(\zeta) = \frac{1}{s!} D^s F(0)(\zeta^s).$ Assume that $a \in B_Y$. Then
\[
\frac{\|D\phi_a(a) P_s(\zeta)\|}{\|D\phi_a(a)\|} \le 1 - \|a\|^2, 
\qquad s \ge 1,\ \zeta \in \mathbb{U},
\]
where $\phi_a \in \mathrm{Aut}(\mathbb{B}_Y)$ satisfies $\phi_a(a)=0$.
\end{lem}
Motivated by Theorem 4.A, we now establish the following improved version involving
the even homogeneous terms of the expansion composed with generalized Schwarz mappings.
This result provides a sharper Bohr-Rogosinski phenomenon under the prescribed assumptions.
\begin{theo}\label{thm:3.3}
Let $B_X$ be the unit ball of a complex Banach space $X$ and $\mathbb{B}_Y$ the unit ball
of a JB$^*$-triple $Y$. Let $F: B_X \to \overline{\mathbb{B}_Y}$ be a holomorphic mapping of the form
\[
F(z) = a + \sum_{s=1}^{\infty} P_s(z), \qquad z \in B_X,
\]
where $P_s(z) = \frac{1}{s!} D^s F(0)(z^s).$ Assume that $a \in \mathbb{B}_Y$. Then for $m, N \in \mathbb{N}$ and $p \in (0,2]$, we have
\[
\|F(\mu_m(z))\|^p + \sum_{s=1}^{\infty} \frac{\|D\phi_a(a) P_{2s}(\mu_k(z))\|}{\|D\phi_a(a)\|}
\le 1,
\]
for $\|z\| = r \le R_{1,k,m,p}$, where $v: B_X \to B_X$ is a Schwarz mapping having
$z=0$ as a zero of order $m$, $\phi_a \in \mathrm{Aut}(\mathbb{B}_Y)$ satisfies $\phi_a(a)=0$,
and $R_{1,k,m,p}$ is the unique root in $(0,1)$ of equation (\ref{e2.2}). Moreover,
$R_{1,k,m,p}$ is sharp.
\end{theo}

\begin{proof}
Fix $z_0 \in \partial B_X$ and define $f(\zeta) = F(\zeta z_0), \zeta \in \mathbb{U}.$
Then $f \in H(\mathbb{U}, \ol{\mathbb{B}_Y})$ and
\[
f(\zeta) = a + \sum_{s=1}^{\infty} P_s(z_0) \zeta^s, \qquad \zeta \in \mathbb{U}.
\]
Since $a = f(0) \in \mathbb{B}_Y$, Lemma~\ref{lem:3.2} implies
\begin{equation}\label{eq:3.1}
\frac{\|D\phi_a(a) P_s(z_0)\|}{\|D\phi_a(a)\|} \le 1 - \|a\|^2, \qquad s \ge 1.
\end{equation}

Let $x = \|a\| \in [0,1)$. For $r \in (0,1)$, using Lemma \ref{L1} together with
estimates (\ref{eq:Schwarz}) and \eqref{eq:3.1}, we obtain
\beas
&&\|F(\mu_m(r z_0))\|^p + \sum_{s=1}^{\infty} \frac{\|D\phi_a(a) P_{2s}(\mu_k(r z_0))\|}{\|D\phi_a(a)\|}\\
&\le & \left(\frac{x + r^m}{1 + x r^m}\right)^p + (1 - x^2) \sum_{s=1}r^{2sk}\\
&=&\left(\frac{x + r^m}{1 + x r^m}\right)^p + (1 - x^2) \frac{r^{2k}}{1 - r^{2k}}\\
&=& 1+\left(\frac{x + r^m}{1 + x r^m}\right)^p-1+ (1 - x^2) \frac{r^{2k}}{1 - r^{2k}}\\
&=&1+D_5(x,r),
\eeas
By Lemma \ref{L4}, we deduce that $D_{5}(x,r)\leq 0$, for $0\leq r\leq R_{1,k,m,p}$,
where $R_{1,k,m,p}$ is the minimal positive root in $(0,1)$ of the equation $p\left(\frac{1-r^{m}}{1+r^{m}} \right)- \frac{2r^{2k}}{1-r^{2k}}=0.$\par

We next verify that the radius $R_{1,k,m,p}$ cannot be improved.

To prove sharpness, fix $z_0 \in \partial B_X$ and $u \in \partial \mathbb{B}_Y$. For
$\lambda \in (0,1)$, define
\[
F(z) = f(\ell_{z_0}(z)) u, \qquad z \in B_X,
\]
where
\[
f(\zeta) = \frac{\lambda - \zeta}{1 - \lambda \zeta}, \qquad \zeta \in \mathbb{U},
\]
and $\ell_{z_0} \in T(z_0)$. Let $\mu_k(z) = - \ell_{z_0}(z)^{m-1} z$. Then
\[
\|F(\mu_m(r z_0))\|^p + \sum_{s=1}^{\infty} \frac{\|D\phi_a(a) P_{2s}(\mu_k(r z_0))\|}{\|D\phi_a(a)\|}
= \left(\frac{\lambda + r^m}{1 + \lambda r^m}\right)^p + (1 - \lambda^2) \frac{\lambda  r^{2k}}{1 - \lambda^2 r^{2k}}.
\]
As in the proof of Theorem \ref{T2.1}, for each $r \in (R_{1,k,m,p}, 1)$ there exists
$\lambda \in (0,1)$ such that
\[
\|F(\mu_m(r z_0))\|^p + \sum_{s=1}^{\infty} \|P_{2s}(\mu_k(r z_0))\| > 1.
\]
Hence, $R_{1,k,m,p}$ is optimal. This completes the proof.
\end{proof}

\begin{rem} In the special case where $B_X = \mathbb{B}_Y = \mathbb{U}$, $\mu_m(z)=\mu_k(z)=z$, and $p=1$, 
Theorem~\ref{thm:3.3} coincides with Theorem~1.A.

\end{rem}
\section*{{\bf Declarations}}

\subsection*{Funding}
The second author acknowledge financial support from the Council of Scientific and Industrial Research (CSIR), New Delhi, India, under Grant Nos. 09/1224(16975)/2023-EMR-I.

\subsection*{Data Availability Statement}
Data sharing is not applicable to this article as no datasets were generated or analyzed during the current study.

\subsection*{Conflict of Interest}
The authors declare that they have no conflict of interest.

\subsection*{Author Contributions}
All authors contributed equally to this work and approved the final manuscript.

\includepdf[pages=1-3]{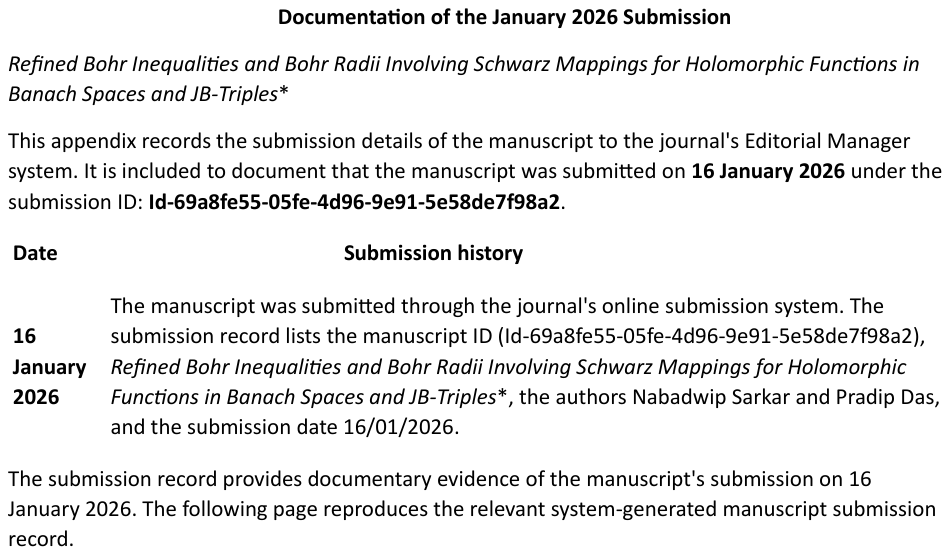}
\end{document}